\documentclass[12pt]{article}
\usepackage[T1]{fontenc}
\usepackage[utf8]{inputenc}
\usepackage{amsmath,amsfonts,amscd,bezier}
\usepackage{amsthm}
\usepackage{import}
\usepackage[normalem]{ulem}
\usepackage{amssymb}
\usepackage{graphicx}
\usepackage[brazil,english]{babel} 
\usepackage[Conny]{fncychap}
\usepackage{hyperref}
\usepackage{fancyhdr} 
\usepackage{indentfirst}
\usepackage{xfrac}
\usepackage{imakeidx} 
\everymath{\displaystyle}
\newtheorem{theorem}{Theorem}[section]
\newtheorem{proposition}[theorem]{Proposition}
\newtheorem{lemma}[theorem]{Lemma}
\newtheorem{corollary}[theorem]{Corollary}
\theoremstyle{definition}
\newtheorem{definition}[theorem]{Definition}
\newtheorem{example}[theorem]{Example}
\newtheorem{remark}[theorem]{Remark}

\usepackage[dvipsnames]{xcolor}
\usepackage{tikz}
\usepackage{multicol}
\usepackage{float}
\usepackage{mathpple}
\usepackage{times}
\usepackage{comment}

\makeindex 

\begin{document}

\title{\bf Divergent diagrams of folds associated with reflections}

\author{Patr\'{\i}cia~H.~Baptistelli, Maria Elenice R. Hernandes\\
{\small Departamento de Matem\'atica, Universidade Estadual de Maring\'a}\\
{\small Av. Colombo, 5790, 87020-900, Maring\'a - PR, Brazil \footnote{Email address: phbaptistelli@uem.br (corresponding author), merhernandes@uem.br}}
\and Miriam Manoel\\
{\small Departamento de Matemática, ICMC, Universidade de São Paulo}\\
{\small 13560-970, Caixa Postal 668, São Carlos - SP, Brazil \footnote{Email address: miriam@icmc.usp.br}} \thanks{The authors are partially supported by CNPq grants 407454/2023-3.}}

\date{}
\maketitle

\begin{abstract}
We analyse divergent diagrams of \(k\)-fold map-germs on \((\mathbb{C}^n,0)\), for $k, n \geq 2$, associated with reflections, adapting to the complex setting the theory of folds associated with involutions on \((\mathbb{R}^n,0)\). In the complex case, a \(k\)-fold is naturally related to a cyclic group generated by a reflection, which guides the analytic classification of singularities. 
Under the conditions of transversality and linearity of the associated reflections, 
certain conditions related to the nontrivial eigenvalues appear as invariants by simultaneous conjugacy. We also provide a complete classification of pairs of transversal linear reflections and the corresponding divergent diagrams.
\end{abstract}



\section{Introduction}

The study of diagrams of smooth mappings is a central topic in differential geometry and singularity theory, particularly in  context of classification of mappings based on their singularities (for example, folds, cusps, etc). A theory for the study of diagrams was established by Dufour in \cite{dufour}, identifying the two distinct cases,  the convergent diagrams and the divergent diagrams. In terms of the classical singularity theory, the convergent case is better established, primarily due to the existence of a Preparation Theorem (which essentially states that locally a holomorphic germ can be written as a polynomial in one variable, up to an invertible factor). In contrast, the absence of a Preparation Theorem in the divergent case makes the study more complex and less systematic. Nevertheless, important contributions to divergent diagrams have been made in some contexts, for example \cite{bivia, Kuro, Mancini, MMT, Teixeira}.

Reflections have been studied and applied from multiple viewpoints, generally considered as linear mappings. They play a central role in the theories of Coxeter and Weyl (\cite{cox, H}). We also cite, for example, the complete classification  of linear reflection groups 
on $\mathbb{C}^n$  by G. Shephard and J. A. Todd \cite{Todd}. In our study, reflections on  $(\mathbb{C}^n,0)$ arise naturally as nonlinear mappings, since they will be associated with fold singularities (see Definition~\ref{def: association}). Specifically, the classification of divergent diagrams of folds on  $(\mathbb{C}^n,0)$ is derived from their association with the corresponding reflections.

More precisely, for $s \leq n$, we establish that two divergent diagrams of $s$ folds are equivalent under the Mather's group $\mathcal{A}=\mathcal{R}\times  \mathcal{L}$ (right-left equivalence) if and only if the product group obtained from the associated $s$-tuples of reflections are conjugate (see Theorem~\ref{teoequivdiagramas2}). Our results generalize the corresponding case for $(\mathbb{R}^n,0)$ considered in \cite{MMT}.

As mentioned above, this work is devoted to divergent diagrams in which the map-germs are fold singularities on $(\mathbb{C}^n,0)$, $n \geq 2$. We discuss the correlation between the classification of divergent diagrams of folds and the classification of associated reflections on $(\mathbb{C}^n,0)$.   
Our motivation is primarily rooted in the study of Mancini, Manoel, and
Teixeira~\cite{MMT}, where the authors investigate the classification of
divergent diagrams of folds associated with involutions on $(\mathbb{R}^n,0)$.
In this context, an involution is a germ of diffeomorphism of order 2 whose fixed-point submanifold may have any codimension between $1$ and $n-1$. In \cite{MMT} the authors address the case where this codimension is one and each fold is associated with a unique involution. In the complex case, a reflection is a germ of biholomorphism on $(\mathbb{C}^n,0)$ of finite order whose fixed-point set is a submanifold of codimension 1. In this case, a $k$-fold is naturally associated with a cyclic group generated by a reflection of order $k$.

In this paper, we explore the implications of this fundamental difference. More importantly, the extension to the complex framework developed here sheds light on certain invariants appearing in the algebraic expressions of the normal forms in the real case, which had remained unexplained for nearly twenty years.

For $k \geq 2$ integer, we study $k$-folds $p:(\mathbb C^n,0)\to (\mathbb C^n,0)$ with respect to their association with reflections $\alpha$, that is, satisfying $p \circ \alpha = p.$  Also, for a vector $\underline{k} = (k_1, \ldots, k_s)$ of integers, we study divergent diagrams 
$$(p_1, \ldots, p_s) : (\mathbb C^n,0) \to (\mathbb C^n \times \cdots \times \mathbb C^n,0),$$
 with respect to their association with  $s$-tuples of reflections $(\alpha_1, \ldots, \alpha_s)$,   
where  $p_i$  is a $k_i$-fold associated with the reflection $\alpha_i$, that is, 
\[ p_i \circ \alpha_i = p_i, \ \ i = 1 \ldots, s.  \]
For short, we shall say that $(p_i)_s$ is associated with $(\alpha_i)_s$. Our main interest is the classification and recognition of diagrams  $(p_i)_s$ from the classification  of the $s$-tuple $(\alpha_i)_s$ of reflections. In particular, following \cite{MMT} we have been driven to the classification of $k$-folds associated with transversal set of linear reflections. \\

This paper is organized as follows. In Section 2, we analyze the interplay between reflections and $k$-folds that will be used throughout the paper. Section 3 is devoted to the case of $k$-folds associated with reflections of order $k$. In particular, we prove that in the class of $k$-fold singularities, the relation of associated with a given reflection is a complete invariant with res\-pect to $\mathcal{L}$-equivalence (Theorem~\ref{thmlequiv}). In Section 4, we study divergent diagrams of $\underline{k}$-folds, where we first establish their equivalence by means of the associated reflections (Theorem \ref{teo: equivdiagrams1}). This is then applied to give explicit normal forms for diagrams of two folds under a linearity and transversality condition of the associated pair of reflections (Theorems \ref{prop: normal-form-pair-reflection} and \ref{teoclassdobrassnm2}).

\section{$k$-folds associated with reflections} 
\label{sec:kfolds}

The aim of this section is to describe the relationship between \(k\)-folds and reflections on \((\mathbb{C}^n,0)\), for \(k, n \geq 2\), and present consequences arising from this connection.

A (complex) reflection is a germ of biholomorphism $\varphi: ({\mathbb C}^n,0)  \to ({\mathbb C}^n,0)$ of finite order $r \geq 2$ whose fixed-point set ${\rm Fix}({\varphi}) = \{x \in ({\mathbb C}^n,0); \ \varphi(x) = x\}$ is a submanifold of codimension 1. For any $j$, if $m \geq 2$ is the order of the reflection $\varphi^{j}$, with
\[\varphi^{j} = \underbrace{\varphi \circ \ldots \circ \varphi}_{\text{$j$ times}} \ , \] 
then $m = \tfrac{r}{gcd(j,r)}$, where $gcd(j,r)$ is the greatest common divisor of $j$ and $r$.  

Let us also recall the Mather's group ${\cal A} = {\cal R} \times {\cal L}$ of pairs $(h,l)$ of germs of biholomorphisms $h,l: (\mathbb{C}^n,0) \to (\mathbb{C}^{n}, 0)$ and its standard action on the set of holomorphic germs $f:  (\mathbb{C}^n,0) \to (\mathbb{C}^n,0)$ given by $(h,l) \cdot f = l \circ f \circ h^{-1}$. For $k \geq 2$, a map-germ 
is called a $k$-fold singularity, or simply a $k$-fold, if it is ${\cal A}$-equivalent to $ f_1: (\mathbb{C}^n,0) \to (\mathbb{C}^n,0) $,   
	\begin{equation} \label{eq:f1}
	 	f_1 (x) = (x_1 ^k , x_2 , \ldots, x_n )
    \end{equation} for  $x = (x_1, \ldots, x_n) \in (\mathbb{C}^n,0)$. 

For simplicity, throughout the paper we suppress the term \emph{germ} and refer to germs of (bi)holomorphisms simply as (bi)holomorphisms.

\begin{definition} \label{def: association}
	Given a reflection $ \varphi$ and a $k$-fold $f$ on $(\mathbb{C}^n,0)$, we say that $f$ is associated with $ \varphi $, and $ \varphi $ is associated with $ f $, if $f \circ \varphi = f $. 
\end{definition}

It is immediate from this definition that if $\varphi$ is associated with $f$, then any reflection in the group $\langle\varphi\rangle$ generated by $\varphi$ is associated with $f$; so  reflections associated with a $k$-fold are not necessarily of order $k$. \\

For the canonical $k$-fold $f_1$ given in (\ref{eq:f1}), we shall denote by $\varphi_1 : (\mathbb{C}^n,0) \to (\mathbb{C}^n,0) $ the canonical linear reflection
\begin{equation}\label{eq:varphi1}
	\varphi_1 (x) = (\lambda x_1 , x_2, \ldots , x_n), 
	\end{equation} 
where $\lambda = e^{\frac{2 \pi i}{k}}$ is the first primitive $k$-th root of unity. Notice that $\varphi_1$ generates a cyclic group $\langle\varphi_1\rangle \simeq {\bf Z}_k$. Following (\ref{eq:f1}) and (\ref{eq:varphi1}), let us also introduce the following natural notation: for $i = 1,\ldots, n$, we denote the $k$-folds
	\begin{equation} \label{eq: fold i}
	 	f_i (x) = (x_1, \ldots, x_i^k , \ldots, x_n )
    \end{equation}
and their associated canonical reflections
\begin{equation}\label{eq: reflection i}
	\varphi_i (x) = (x_1, \ldots, \lambda x_i,  \ldots , x_n). 
	\end{equation} 

\noindent For the singular set $\Sigma(f_i)$ of $f_i$,  we obviously have 
\begin{equation} \label{eq: sigma 1 = fix 1}
\Sigma(f_i) = {\rm Fix}({\varphi_i}), \ \ \  i = 1, \ldots, n.
\end{equation}

We now define the equivalence relation on the set of reflections. 

\begin{definition} \label{def:uplaequiv}
Two reflections $\alpha, \beta: (\mathbb{C}^n,0) \to (\mathbb{C}^n,0)$  are equivalent if they are conjugate, namely if $\beta = h \circ \alpha \circ h^{-1}$, where $h$ is a biholomorphism  on $(\mathbb{C}^n,0)$. 
\end{definition}

Next, we present a sequence of results on  $k$-fold singularities and reflections on $(\mathbb{C}^n,0)$, extending the results of \cite{MMT} for folds and involutions on $(\mathbb{R}^n,0)$. For the results regarding a canonical $k$-fold $f_i$ as in \eqref{eq: fold i} and its canonical associated reflection $\varphi_i$ as in (\ref{eq: reflection i}), we fix $i=1$ for simplicity.

\begin{lemma}\label{obsphizero}
For $f_1$ given in (\ref{eq:f1}), any reflection associated with $f_1$ belongs to the group $\langle \varphi_1\rangle$, with $\varphi_1$ given in (\ref{eq:varphi1}).    
\end{lemma}

\begin{proof} Let $ \rho: (\mathbb{C}^n,0) \to (\mathbb{C}^n,0)$ be a reflection such that $f_1 \circ \rho = f_1$. Given in its coordinate functions $\rho =  (\rho_1,  \ldots , \rho_{n})$, for $x=(x_1,\ldots, x_n) \in (\mathbb{C}^n,0)$, we have $ (\rho_{1} (x))^k = x_1 ^k$, and so $\rho_1(x) = \lambda^j x_1$,  with $\lambda = e^{\frac{2 \pi i}{k}}$ and $j \in \{1, \ldots, k-1\}$. Also, $\rho_{i} (x)= x_i $, for all $  i \in \{2,\ldots, n  \}$. Hence, $\rho  \in \langle \varphi_1\rangle$.
\end{proof} 

The following result generalizes the above lemma. 
\begin{proposition}\label{prop:reflgroup} 
For a $k$-fold $f$ and $\varphi_1$ given in (\ref{eq:varphi1}), any reflection associated with $f$ belongs to $h \langle \varphi_1\rangle h^{-1}$, where $h$ is such that $f = (h,l)\cdot f_1$, with $(h,l) \in {\cal R} \times {\cal L}$.     
\end{proposition}
\begin{proof} Trivially, elements in $h \langle \varphi_1\rangle h^{-1}$ are reflections associated with $f$. Now, let $\rho$ be a reflection such that $f \circ \rho = f$. Then $f_1 \circ(h^{-1} \circ \rho \circ h) =f_1$, which, together with Lemma~\ref{obsphizero}, implies that $\rho \in h \langle \varphi_1\rangle h^{-1}$. 
\end{proof}

\begin{example}
\label{Ex1-Fold-Ref}
We  compute the reflection group of the $4$-fold $f: (\mathbb{C}^3,0) \to \mathbb{C}^3,0)$, 
\begin{equation} \label{eq: example}
f(x_1,x_2,x_3)=\big((x_1+x_3)^2-x_3^4,\ x_3^4+x_1+x_3, \ 2x_2 \big).
\end{equation}
We have $f= (h,l) \cdot f_1$, where $f_1$ is given in (\ref{eq:f1}), with $l(x_1,x_2,x_3)=(x_2^2-x_1,x_1+x_2,2x_3)$ and $h(x_1,x_2,x_3)=(x_1+x_2, x_3, -x_1)$. For $\varphi_1$ as in (\ref{eq:varphi1}) with $\lambda = i$, it follows from Proposition~\ref{prop:reflgroup} that the group generated by $\varphi = h \circ \varphi_1 \circ h^{-1}$, $$\varphi(x_1,x_2,x_3)=(x_1+(1-i)x_3,\ x_2, \ ix_3),$$ is the group of  reflections associated with $f$.  	   
\end{example}

The next result reveals how the orbit of $k$-fold singularities under ${\cal A}$-equivalence is stratified as a set of $k$-folds 
associated with reflections.

\begin{proposition}
\label{cor:groupreflections}
Let $f$ and $g$ be $k$-fold singularities with $g=(h,l)\cdot f$, for $(h,l) \in {\cal R} \times {\cal L}$. Then the groups of reflections associated with $f$ and $g$ are conjugate by $h$; more precisely, if $\langle \varphi \rangle$ is the group of reflections associated with $f$, then  $ h \langle \varphi \rangle h^{-1}$ is the group of reflections associated with  $g$. 
\end{proposition}
\begin{proof}
It follows from Proposition~\ref{prop:reflgroup} that  the group of reflections associated with $f$ is  $h_1\langle \varphi_1 \rangle h_1^{-1}$, for some $h_1 \in {\cal R}$. Applying now this proposition to $g$, it follows that 
\[ (h \circ h_1) \langle \varphi_1 \rangle (h\circ h_1)^{-1} = h \langle h_1 \varphi_1 h_1^{-1} \rangle h^{-1} \] is the group of reflections associated with $g$, which gives the result. 
\end{proof}
We now present another consequence of Proposition~\ref{prop:reflgroup}, which generalizes (\ref{eq: sigma 1 = fix 1}).

\begin{proposition}  \label{prop: sigma=fix}
	Let $f:(\mathbb{C}^n,0)\to (\mathbb{C}^n,0)$ be a $k$-fold associated with a reflection $\varphi$ on $(\mathbb{C}^n,0)$. Then $\Sigma (f) = {\rm Fix}(\varphi).$
\end{proposition}

A natural question raised here is whether, for a given reflection $\varphi$ of order $k \geq 2$, there exists a $k$-fold associated with $\varphi$. We have:

\begin{proposition} \label{prop: existence of fold}
	Given a reflection $ \varphi$ on $(\mathbb{C}^n, 0)$ of order $k \geq 2$, there is a $k$-fold associated with $ \varphi$.
\end{proposition}
\begin{proof} Every reflection  $\varphi$ is conjugate to its linear part $d \varphi(0)$. In fact, consider the averaging conjugacy $H: (\mathbb{C}^n,0) \to (\mathbb{C}^n,0)$, 
$$H = \frac{1}{k} \sum_{j=0}^{k-1} (d\varphi(0))^{-j} \circ \varphi^j,$$ 
to obtain  $H \circ \varphi = d\varphi(0) \circ H$. Since $d\varphi(0)$ is also a reflection of order $k$, its canonical form is $\varphi_1^j$, for some $j \in \{1, \ldots, k~-~1\}$.  Hence there exists $h \in {\cal R}$ such that $\varphi = h \circ \varphi_1^j \circ h^{-1}$, and  $ f= f_1 \circ h^{-1}$ is a $k$-fold associated with $\varphi$. \end{proof}
Now, we combine Proposition~\ref{cor:groupreflections} and Proposition~\ref{prop: existence of fold} to obtain the following: 

\begin{theorem} \label{thm: realisation}
Given $k \geq 2$, any cyclic reflection group of order $k$ is realised as the group of reflections associated with a $k$-fold singularity. 
\end{theorem} 

Based on Theorem \ref{thm: realisation} above, throughout the remainder of the paper we use the term reflection group to refer to a cyclic group generated by a reflection.

\section{$k$-folds associated with reflections of order $k$} 

In this section, we refine the stratification obtained in Proposition~\ref{cor:groupreflections}. For a fixed reflection of order $k \geq 2$, we prove that the orbit of $k$-fold singularities associated with it under ${\mathcal A}$-equivalence reduces to $\mathcal L$-equivalence (Theorem~\ref{thmlequiv} and Corollary~\ref{thm:groupequiv}).   In other words, for any reflection $\varphi$ on $(\mathbb{C}^n,0)$, being associated with $\varphi$ is clearly an invariant under $\mathcal L$-equivalence; we show that if in addition  $\varphi$ has order $k$, then this invariant is in fact complete inside the class of $k$-fold singularities. 

As a preliminary step, we establish the following lemma.

\begin{lemma}\label{lemma:mirianarrasa}
Let $ l : (\mathbb{C}^n, 0) \rightarrow (\mathbb{C}^n, 0) $ be a holomorphism such that $ g = l \circ f_i$ is a $k$-fold, for some $i = 1, \ldots, n$, with $f_i$  in \eqref{eq: fold i}. Then $l$ is a biholomorphism.
\end{lemma}
\begin{proof} To ease exposition, we choose $i=1$. Write $ x = x_1$, $y = ( x_2 , \ldots, x_n)$ and $l = (dl)_0 + R ,$ where $R$ is the remainder of order greater than 1. In matricial form,  
$$(dg)_0 = (d l)_0 (df_1)_0 =  
		\left(
		\begin{array}{c|c}
			a & B\\
			\hline
			C & D
		\end{array} 
		\right)
		\left(
		\begin{array}{c|c}
			0 & 0\\
			\hline
			0 & I_{n-1}
		\end{array} 
		\right) =
		\left(
		\begin{array}{c|c}
			0 & B\\
			\hline
			0 & D
		\end{array} 
		\right),$$ for $a \in \mathbb{C}$ scalar and $B, \ C,\ D$  matrices. 
        
        Without loss of generality we can assume that $D$ is invertible, for $ g $ is a $k$-fold. 
 Hence, rank$(dl)_0$  $ = n-1$ if and only if $a - BD^{-1}C = 0$. Let us prove that $a - BD^{-1}C \neq 0$.
 
Consider the linear isomorphism whose matrix is given by \begin{equation*}
		L = \left(
		\begin{array}{c|c}
			1 & -BD^{-1} \\ \hline 0  &  D^{-1}
		\end{array}
		\right).
	\end{equation*}
	
	Then
    \[	L \circ l \circ f_1(x,y)
		= L \circ l(x^k, y)= 
		\left(
		\begin{array}{c|c}
			a - BD^{-1}C & 0  \\
			\hline
			  D^{-1}C &  I_{n-1}
		\end{array}
		\right)
		\begin{pmatrix}
			x^k\\
			y
		\end{pmatrix}
		+
		\begin{pmatrix}
			U_1 (x, y)\\
			U_2 (x,y)
		\end{pmatrix},\]
	\noindent  that is,  
    $$L \circ l \circ f_1  (x,y) = ((a-BD^{-1}C ) x^k + U_1 (x,y), \alpha (x,y)),$$ 
where $\alpha( x, y ) = D^{-1}C x^k + y + U_2 (x,y),$ 
$ U_1:(\mathbb{C}^n, 0) \rightarrow (\mathbb{C}, 0)$ and $ U_2: (\mathbb{C}^n, 0) \rightarrow (\mathbb{C}^{n-1}, 0) $ with terms of order greater than $2k$ in $x$ and greater than or equal to $1$ in $y$. We have that  $\phi (x,y) = (x,  \alpha (x,y))$ defines a biholomorphism on $(\mathbb{C}^n, 0)$ and,  moreover, $\alpha (\phi^{-1} ( x, y )) = y.$ Therefore,
	$$L \circ l \circ f_1 \circ \phi^{-1} (x,y)  =  ((a-BD^{-1}C)x^k, y) + R'(x,y),$$ where $R'(x,y) = (U_1(\phi^{-1} (x,y)),0)$ has terms of order greater than or equal to $2k$ in $x$. Now,  $L \circ l \circ f_1 \circ \phi^{-1}$ being ${\cal A}$-equivalent to $f_1$ implies that $a-BD^{-1} C \neq 0$. 
    
    Finally,  
   $ \mathrm{rank}(dl)_0 \geq \mathrm{rank}((dl)_0 d(f_1)_0) = \mathrm{rank}(dg)_0  = n-1,$ 
    and this completes the proof.
 \end{proof}

\begin{proposition}\label{propequivnsein2} Consider $k \geq 2$, $ f_1 $ and $\varphi_1$ as in (\ref{eq:f1}) and (\ref{eq:varphi1}), respectively. Let $g$ be a $k$-fold and $j \in \{ 1, \ldots, k-1\}$, with $gdc(j,k) = 1$. Then,
\begin{itemize}
    \item[(a)] $g$ is associated with $\varphi_1^j$ if and only if $g$ is $ \mathcal{L} $-equivalent to $ f_1 $;
    \item[(b)] $ g $ is associated with a reflection $ \varphi \in h \langle \varphi_1 \rangle h^{-1}$ of order $k$, for $h \in {\cal R}$, if and only if $g$ is $ \mathcal{L} $-equivalent to $ f_1 \circ h^{-1} $. 
\end{itemize}
\end{proposition}

\begin{proof} $(a)$ If $ g $ is a $k$-fold associated with $ \varphi_1^j$, where $gcd(j,k) = 1$, then $g \circ \varphi_1 = g$. Each coordinate function $g_m : (\mathbb{C}^n, 0)\rightarrow  (\mathbb{C},0) $ of $g$ satisfies $g_m \circ \varphi_1 = g_m$, for $m \in \{1,  \ldots, n\}$. Hence, for all $x = (x_1, \ldots x_n)$, 
$$g_m (x)=  l_m  (x_1^k, x_2,\ldots, x_n) =  l_m \circ f_1 (x),$$
for some holomorphism $l_m : (\mathbb{C}^n, 0) \rightarrow (\mathbb{C},0)$.  Now, just define $l~=~(l_1,~\ldots,~l_n)$, and  then
	$g =  ( g_1 , \ldots, g_n ) = (l_1 \circ f_1 , \ldots, l_n \circ f_1) =  l \circ f_1.$ By Lemma \ref{lemma:mirianarrasa}, $g$ is $ \mathcal{L} $-equivalent to $f_1$. The converse is immediate. 
    
$(b)$ \  Let $ g $ be a $k$-fold associated with $ \varphi = h \circ \varphi_1^j \circ h^{-1}$, with $gcd(j,k) = 1$. Then  $(g\circ h) \circ \varphi_1^j = g \circ h$ and from $(a)$  there exists $l \in {\cal L}$ such that $ g = l \circ (f_1 \circ h^{-1})$. The converse is immediate.   
        \end{proof}

\vspace{.3cm}

We now generalize Proposition~\ref{propequivnsein2}$(a)$, showing that the ${\cal A}$-equivalence class of $k$-folds associated with reflections of the same order $k$  is reduced to ${\cal L}$-equivalence:

\vspace{.2cm}

\begin{theorem}
\label{thmlequiv}
For $k \geq 2$, let $f$ be a $k$-fold associated with a reflection $\varphi$ of order $k$. Then  a $k$-fold $g$ is associated with $ \varphi$ if and only if $g$ is $ \mathcal{L}$-equivalent to $ f $.
\end{theorem}

\begin{proof} Write $f = (h,l)\cdot f_1$, for $(h,l) \in {\cal R} \times {\cal L}$. If $f$ is associated with $\varphi$, then by Proposition \ref{prop:reflgroup} we have $\varphi = h \circ \varphi_1^j \circ h^{-1}$ for some $j \in \{1,\ldots, k-1\}$. By Proposition \ref{propequivnsein2}$(b)$, if $g$ is a $k$-fold associated with $\varphi$, then $g = l' \circ (f_1 \circ h^{-1})$ for some $ l' \in {\cal L} $. In this case, $g = (l' \circ l^{-1}) \circ f$. Therefore, $g$ is $\mathcal{L}$-equivalent to $f$. The reciprocal of the result is straightforward.\end{proof} 

\vspace{.3cm}

\begin{example}
Consider the $4$-fold $g = f \circ h^{-1}: (\mathbb{C}^3,0) \to \mathbb{C}^3,0)$,  
$$g(x_1,x_2,x_3)= \big( (x_2-x_3)^2 - x_2^4,\ x_2^4 + x_2 - x_3, \ 2(x_1+x_3) \big),$$ and the reflection $\varphi$ associated with $f$,
where $f,h$ and $\varphi$ are given in Example \ref{Ex1-Fold-Ref}. 
By Theorem \ref{thmlequiv}, $f$ and $g$ are not $\mathcal{L}$-equivalent, since $\varphi$ is not associated with $g$. In fact, $(g \circ \varphi)(x_1,x_2,x_3)=((x_2-ix_3)^2-x_2^4,\ x_2^4+x_2-ix_3,\ 2(x_1+x_3)),$ so $g\circ \varphi \neq g$.
\end{example}

The following is just the Theorem \ref{thmlequiv} stated for reflection groups.

\begin{corollary}
  \label{thm:groupequiv}
For $k \geq 2$, let $f$ be a $k$-fold associated with a reflection group of order $k$. Then $g$ is also a $k$-fold associated with this group if and only if $g$ is $ \mathcal{L}$-equivalent to $ f $.
\end{corollary}

\section{Divergent diagrams of $\underline k$-folds}

We develop in the remainder of the paper the framework and classification of 
divergent diagrams of folds on $(\mathbb{C}^n,0)$. As expected, our results recover those of \cite{MMT} on the 
classification of divergent diagrams of folds and involutions on $(\mathbb{R}^n,0)$. 
More importantly, as mentioned in the Introduction, the extension to the complex setting developed here clarifies 
some invariants appearing in the algebraic expressions of the normal forms in 
the real case, which had remained unexplained for about twenty years.\\

For $s \leq n$ and  $k_i \geq 2$, \ $i = 1, \ldots, s$, we deal with divergent diagrams of $k_i$-folds $ g_i$ on $(\mathbb{C}^n, 0)$, denoted by $(g_i)_s : (\mathbb{C}^n, 0)  \rightarrow (\mathbb{C}^{ns}, 0) $, where
	\[
		(g_i)_s (x) = (g_1 (x), \ldots, g_s (x)). \] 
 In this setting, we shall say that  $(g_i)_s$ is a divergent diagram of ${\underline k}$-folds, where ${\underline k}=(k_1, \ldots, k_s)$ is the vector of the orders. 
Now, let 
\[(\beta_i)_s = (\beta_1, \ldots, \beta_s)\]
denote an $s$-tuple of reflections  on  $(\mathbb{C}^n, 0)$ and consider the direct product of the cyclic reflection groups 
\[ G = \langle \beta_1\rangle \times \cdots \times \langle \beta_s \rangle.\] 
For short, we shall refer to $G$ as the product group $\textstyle\prod_{i=1}^s \langle \beta_i\rangle$.

If $r_i$ is the order of the reflection $\beta_i$, $i = 1, \ldots, s$, then $G$ is a group with $\textstyle\prod_{i=1}^s r_i$ elements. Recall that $G$  is cyclic if and only if gcd($r_i, r_j$) =1 for all $i \neq j$.  In this case, $(\beta_i)_s $ is a generator of $G$, whose order is 
\[
\operatorname{ord}(\beta_1,\ldots,\beta_s)
= \operatorname{lcm}(r_1,\ldots,r_s),
\]
which coincides with the order of the group. 

From the terminology of the previous sections regarding the association between $k$-folds and 
(groups of) reflections, the notion of divergent diagrams $(g_i)_s$ associated with an 
$s$-tuple $(\beta_i)_s$ of reflections is clear, and we shall use  
interchangeably the terminologies:
\begin{itemize}
    \item the divergent diagram of $\underline{k}$-folds $(g_i)_s$ is 
associated with the $s$-tuple $(\beta_i)_s$,
\item the $s$-tuple $(\beta_i)_s$ is associated 
with the divergent diagram of $\underline{k}$-folds $(g_i)_s$, 
\item the divergent diagram of $\underline{k}$-folds  $(g_i)_s$ is associated with the product group 
$G = \textstyle\prod_{i=1}^s \langle \beta_i\rangle$.\\
\end{itemize}

The equivalence relation of $s$-tuples of reflections is given by simultaneous conjugacy.

\begin{definition} \label{def:uplaequiv2}
The $s$-tuples of reflections $(\alpha_i)_s $ and $ (\beta_i)_s $ are equivalent if there exists $h \in \cal{R}$ such that $\beta_i = h \circ \alpha_i \circ h^{-1}$,  $i = 1, \ldots,s$.
\end{definition}

Clearly, if $\alpha_i$ and $\beta_i$ are of the same order, then $(\alpha_i)_s$ and $ (\beta_i)_s$ are equivalent if and only if $\langle \alpha_i\rangle$ and $\langle \beta_i\rangle$ are conjugate, and in this case we say that the product groups $P = \textstyle\prod_{i=1}^s \langle \alpha_i\rangle$ and  $G = \textstyle\prod_{i=1}^s \langle \beta_i\rangle$ are conjugate. \\

We now define the equivalence of divergent diagrams of $\underline{k}$-folds:
    
\begin{definition}\label{defequivdiagramdivdobras} Two divergent diagrams of $\underline{k}$-folds $ (p_i)_s $ and $(g_i)_s$ are equiva\-lent if there exist $l_1 ,\ldots, l_s \in {\cal L}$ and $ h \in {\cal R}$ such that $g_i = l_i \circ p_i \circ h^{-1},$  $i = 1, \ldots,s$. 
\end{definition}

\vspace{.3cm}

The following theorem is the first result relating the two previous equivalences: 

\begin{theorem}\label{teo: equivdiagrams1}
Let $ (p_i)_s $ and $(g_i)_s$ be divergent diagrams of $\underline{k}$-folds associated with the $s$-tuples of reflections $ (\alpha_i)_s $ and $ (\beta_i)_s$, respectively, where $\alpha_i$ and $\beta_i$ have order $k_i \geq 2$, for $i \in \{1, \ldots, s\}$. If $ (\alpha_i)_s $ and $ (\beta_i)_s $ are equivalent, then $ (p_i)_s $ and $ (g_i)_s$ are equivalent. \end{theorem}

\begin{proof} Consider $h \in {\cal R}$ such that $\beta_i = h \circ \alpha_i \circ h^{-1},$  $ i \in \{1, \ldots, s\}$. 
The condition $g_i \circ \beta_i = g_i$ implies that  $(g_ i \circ h) \circ \alpha_ i = g_i \circ h$. Hence, each $k_i$-fold $g_i \circ h$ is associated with $\alpha_i$.  Since $\alpha_i$ has order $k_i$ and $p_i$ is a $k_i$-fold associated with $\alpha_i$, by Theorem \ref{thmlequiv} we have that $g_i \circ h$ and $p_i$ are ${\cal L}$-equivalent, and the result follows. 
\end{proof}

\vspace{.3cm}

The next result shows that the necessary condition of Theorem~\ref{teo: equivdiagrams1} is also sufficient at the group level:

\vspace{.2cm}
 \begin{theorem}\label{teoequivdiagramas2}
Let $ (p_i)_s $ and $ (g_i)_s$  be divergent diagrams of $\underline{k}$-folds associated with  product groups $P = \textstyle\prod_{i=1}^s \langle \alpha_i\rangle$ and $G = \textstyle\prod_{i=1}^s \langle \beta_i\rangle$, respectively, where $\alpha_i$ and $\beta_i$ have order $k_i \geq 2$. Then $ (p_i)_s $ and $ (g_i)_s$ are equivalent if and only if $P$ and $  G$ are conjugate. \end{theorem}

\begin{proof} Let $ (p_i)_s $ and $ (g_i)_s $ be equivalent divergent diagrams of $\underline{k}$-folds associated with the groups $P$ e $G$, respectively. By Proposition~\ref{cor:groupreflections}, the groups $\langle \alpha_i \rangle$ and $\langle \beta_i \rangle$ are conjugate, $i \in \{1, \ldots, s\}$, and so $P$ and $G $ are conjugate. The converse follows directly from Theorem~\ref{teo: equivdiagrams1}. 
\end{proof}

\subsection{Normal forms under transversality}
\label{sub:normalforms}

The aim of this subsection is to derive normal forms for divergent diagrams and their corres\-ponding tuples of reflections under the standard transversality and linearity assumptions. We start by recalling that, for $ s\leq n $,
a set $\Lambda_s = \{\alpha_ 1, \ldots , \alpha_ s \}$ of reflections on $(\mathbb{C}^n,0)$ is transversal if 
\begin{itemize}
\item ${\rm Fix}({\alpha_i})$ is transversal to ${\rm Fix }({\alpha_j})$ at $0$ for $i \neq j$;
\item $ {\rm codim}  \ {\bigcap_{i = 1 }^s {\rm T_0Fix }({\alpha_i})} = \sum_{i= 1}^s {\rm codim}({\rm Fix }({\alpha_i})) = s,$ 
\end{itemize} where ${\rm T_0Fix }({\alpha_i})$ denotes the tangent space to ${\rm Fix }({\alpha_i})$ at $0$. Generally speaking, this means that the fixed-point hyperplanes intersect in the best way possible, in the sense that they do not share extra directions: the normal directions to the fixed-point subspaces  are linearly independent, so no reflection is `redundant' with respect to the others.

A benefit of the transversality condition is that it permits the simultaneous `straightening'  of the fixed-point subspaces of a transversal family of reflections to the canonical hyperplanes by means of simultaneous conjugacy. This in turn provides a simple algebraic expression for a pre-normal form,  given in (\ref{eqpsi}) in the next proposition. 

\vspace{.2cm}

\begin{proposition}     \label{prop:equivalenciadereflexões}
Let $\Lambda_s = \{\alpha_ 1 , \ldots , \alpha_ s \}$ be a transversal set of linear reflections of orders $k_1, \ldots, k_s \geq 2$, respectively.  Let $\lambda_i$ be the eigenvalue of $\alpha_i$, with $\lambda_i \neq 1$. Then  $(\alpha_i)_s$ is equivalent to $(\beta_i)_s$, where
	\begin{equation}\label{eqpsi}
		\beta_ i (x_1, \ldots , x_n ) = (x_1 + a_{i1} x_i, \ldots , \lambda_i x_i, \ldots ,x_n + a_{in} x_i ), 
	\end{equation}
	for $a_{i j} \in \mathbb{C}$,  $ 1 \leq i \leq s,$ \ $1 \leq j \leq  n $,  $ j \neq i $.
\end{proposition}

    \begin{proof} For $i = 1, \ldots, s$, we write the fixed-point hyperplanes ${\rm Fix}({\alpha_ i}) \subseteq (\mathbb{C}^n,0)$ as 
	\begin{equation}\label{eqball}
		{\rm Fix}({\alpha_i }) = \ker (F_i ),
	\end{equation} for linear functionals  $F_i  : (\mathbb{C}^n,0)  \to (\mathbb{C}, 0)  $. These 
    are linearly independent by transversality of $\Lambda_s$ and we complete this set to have a basis $\mathcal{B}$
of $ (\mathbb{C}^n )^*$. Consider the linear isomorphism $h$  sending the dual basis of $\mathcal{B}$
 to the canonical basis of $\mathbb{C}^n$. 
 For $i = 1, \ldots, s$, consider $\beta_i = h \circ \alpha_i \circ h^{-1}$, which is obviously a linear reflection of order $k_i$ whose eigenvalues coincide with the eigenvalues of $\alpha_i$. From (\ref{eqball}),
 $$			{\rm Fix}({\beta_i}) = h({\rm Fix}({\alpha_i})) = \{x=(x_1,\ldots, x_n) \in \mathbb{C}^n;\ x_i =0\},$$ for $ i = 1 , \ldots, s$. Hence, $\beta_i (e_j ) = e_j,$ for $j \neq i$. Now we just write $	\beta_i (e_i ) = ( a_ {i 1}, \ldots, a_ {i n})$ to get 
 $			\beta_i (x) = (x_ 1 + a_ {i 1} x_ i, \ldots, a_ {ii}  x_ i , \ldots,  x_ n + a_ {i n} x_ i)$.  But  $a_{ii}$ is the only eigenvalue of $\beta_i$ distinct from 1, so $a_{ii} = \lambda_i$ and, therefore, $\beta_i$ has the form given by (\ref{eqpsi}).
            \end{proof}

            \quad

In the next result, we solve the recognition problem  of transversal $s$-tuples of linear reflections based on the parameters that appear in the pre-normal form (\ref{eqpsi}) of Proposition \ref{prop:equivalenciadereflexões}.	

\begin{theorem}\label{thm: matrix} Let $\lambda_i$ be a primitive $k_i$-th root of unity. Consider the $s$-tuples of transversal linear reflections $(\beta_{i_a})_s$ and $(\beta_{i_b})_s$, where
\[
	     \beta_{i_a} (x) = (x_1  + a_{i1} x_i, \ldots , \lambda_i x_i, \ldots ,x_n  + a_{in} x_i )\] and \[
     \beta_{i_b} (x) = (x_1  + b_{i1} x_i, \ldots , \lambda_i x_i, \ldots ,x_n  + b_{in} x_i )
	\] have order $k_i \geq 2$, 
	for all $i \in \{1, \ldots, s\}$, with $ a_{ij}, b_{ij} \in \mathbb{C}  $ and $j \neq i $. Then $(\beta_{i_a})_s$ is equivalent to $(\beta_{i_b})_s$ if and only if there exists an invertible  matrix $C$ with elements in $\mathbb{C}$ given by 
    
	\begin{equation}
	    \label{Matrix-H}
        C =\left(
			\begin{array}{cccc|ccc}
			  \xi_1 & 0 & \cdots & 0 & 0 & \cdots & 0  \\
				0 & \xi_2 & \cdots & 0 & \vdots & \vdots  & \vdots  \\
				\vdots & \vdots & \ddots & \vdots & \vdots & \vdots  & \vdots   \\
				0 & 0 & \cdots  & \xi_s & 0 & \cdots  & 0  \\\hline
				c_{s+1,1} & c_{ s + 1,2} & \cdots  & c_{ s + 1 , s 	}  & c_{s+ 1 , s+1} & \cdots & c_{s+1 , n}   \\
				\vdots & \vdots & \ddots & \vdots & \vdots & \ddots  & \vdots   \\
				c_{n1} & c_{n2}  & \cdots  & c_{n s} & c_{n ,s+1} & \cdots & c_{n n}  \\
			\end{array}
			\right)
	\end{equation} 
    such that $\xi_1 = 1$ and 
    $$b_{ij}=\left\{\begin{array}{cl}
     \dfrac{\xi_j }{\xi_i } a_{ij}, & \mbox{for all} \ \ 1  \leq i, j \leq s, \ \mbox{with}\ j \neq i \\     
     (\lambda_1 - 1) c_{j1} + \displaystyle \sum_{k = 2} ^n a_{1 k} c_{jk}, & \mbox{for} \ \ i = 1 \ and \ s+1 \leq  j \leq n \\
	\dfrac{1}{\xi_i}((\lambda_i - 1) c_{ji} + \sum_{\substack{k = 1 \\ k \neq i}} ^n a_{i k} c_{jk}) , & 
	\mbox{for} \ \  2 \leq i \leq s \ \mbox{and} \ s+1 \leq  j \leq n.
    \end{array} \right.$$
\end{theorem}

\begin{proof} By definition, $(\beta_{i_a})_s$ is equivalent to $(\beta_{i_b})_s$ if and only if there exists an invertible matrix $C = (c_{jk})_{n \times n }$ such that \begin{equation}\label{eqfourballs}
		[\beta_{i_b}] C  =  C [\beta_{i_a}]	\end{equation} 
for all $i \in \{1, \ldots, s\}$. If we set \begin{equation*}
  \label{Exp-matrices}
   d_{jl}=\left\{\begin{array}{cc}
     c_{jl}+ c_{il} b_{ij},  & \mbox{for} \ j \neq i \\
     \lambda_i c_{il},  & \mbox{for} \ j=i
  \end{array} \right. \ \ \mbox{and} \ \      
  e_{jl}=\left\{\begin{array}{cc}
     c_{jl},  & \mbox{for} \ l \neq i \\     \displaystyle\sum_{\substack{k = 1 \\ k \neq i}}^n a_{ik}c_{jk}+\lambda_ic_{ji},  & \mbox{for} \ l=i,
  \end{array}\right.
  \end{equation*}
then  $[\beta_{i_b}] C = (d_{jl})_{n\times n}$ and $ C [\beta_{i_a}] = (e_{jl})_{n\times n}$. For  $i \in \{1, \ldots, s\}$ and $j \in \{ 1, \ldots, n\}$, it follows from (\ref{eqfourballs}) that $c_{ij} = 0$ for $j \neq i$. Since $C$ is invertible, we have $\xi_i := c_{ii} \neq 0$, and clearly we can assume $\xi_1 = 1 $.
For $i, j \in \{1, \ldots, s\}$, with $ j \neq i $, 
	\[
		c_{j i} + \xi_i b_{ij} = \sum_{\substack{k = 1 \\ k \neq i}} ^n a_{ik} c_{jk} + \lambda_i c_{ji},
	\]
    with $c_{ji} = 0$. Hence,  $b_{ij} = \dfrac{\xi_j}{\xi_{i}} a_{ij},$ for all $i, j \in \{1, \ldots, s\}, \ j \neq i .$ 
    
    For $i = 1$ and $s+1 \leq j \leq n$, we obtain     
	$\sum_{k = 2} ^ n  a_{1k} c_{j k} + \lambda_1 c_{j1 } = c_{j 1} + \xi_{1} b_{1 j},$ that is,
    \[ b_{1j} =(\lambda_1 - 1)c_{j1} + \sum_{\substack{k = 2}} ^ n  a_{1k} c_{j k}. \] 
Similarly, for $2 \leq i \leq s $ and $s+1 \leq j \leq  n$, we obtain 
$		\sum_{\substack{k = 1 \\ k \neq i}} ^ n  a_{ik} c_{j k} + \lambda_i c_{ji }  
		=  c_{j i } + \xi_i b_{ij},$
        that is,  
        \[ b_{ij} 
		=\dfrac{1}{\xi_i}((\lambda_i - 1) c_{ji} + \sum_{\substack{k = 1 \\ k \neq i}} ^n a_{i k} c_{jk}). \] 
        \end{proof}
We close this subsection by presenting pre-normal forms of divergent diagrams.
From now on, let $f_i$ denote the canonical $k_i$-fold  
        \begin{equation}\label{eq:classicfold}
		f_i(x) = (x_1, \ldots, x_ i^{k_i}, \ldots , x_n), \end{equation}
consistent with the notation of \eqref{eq: fold i} in Section~\ref{sec:kfolds}, where only a single fold is considered. Similarly, we update the notation of $\varphi_i$ 
 introduced in (\ref{eq: reflection i}) with $\lambda = \lambda_i$, to denote the reflection of order $k_i$
  \begin{equation}\label{eq:Abelian}
		\varphi_{i}(x) = (x_1, \ldots, \lambda_i x_ i, \ldots , x_n).\end{equation}  

 \quad
            
\begin{proposition}\label{prop:equivcomreflexoes}
	Let $\Lambda_s = \{\alpha_ 1, \ldots, \alpha_ s \}$ be a transversal set of linear reflections of orders $k_1, \ldots, k_s \geq 2$, respectively. Let $\lambda_i$ be the eigenvalue of $\alpha_i$, with $\lambda_i \neq 1$.  A divergent diagram of $\underline{k}$-folds $(p_i)_s$ associated with $(\alpha_i)_s$ is equivalent to the divergent diagram of $\underline{k}$-folds $(g_i)_s$ associated with $(\beta_i)_s$, where $\beta_i $ is given in (\ref{eqpsi}) and\begin{equation}\label{eq: normalform_g}
		g_i(x) = \bigl(x_ 1  + \dfrac{a_ {i1}}{1-\lambda_i} x_ i, \ldots , x_ i ^{k_i} , \ldots ,x_ n  + \dfrac{a_ {in}}{1-\lambda_i} x_ i \bigr),
	\end{equation} for $a_{i j} \in \mathbb{C}$, with $ i \in \{1 , \ldots, s\}$, $j \in \{ 1, \ldots, n \}$ and $ j \neq i $. 
\end{proposition}

\begin{proof} By Proposition \ref{prop:equivalenciadereflexões}, $ ( \alpha_i)_s $ is equivalent to $ (\beta_i)_s $, for $\beta_i$ defined in \eqref{eqpsi}.  
For each $i \in \{ 1, \ldots , s\} $, consider $g_i$ as in \eqref{eq: normalform_g} and define $ h_i: (\mathbb{C}^n,0) \to (\mathbb{C}^n,0)$ 
by
\[
		h_i (x) =  \bigl(x_ 1 - \dfrac{a_ {i1}}{1-\lambda_i} x_ i , \ldots , x_ i , \ldots , x_ n - \dfrac{a_ {in}}{1-\lambda_i} x_ i \bigr).
\] 
For the reflection $\varphi_i$ and the $k_i$-fold $f_i$ in (\ref{eq:Abelian}) and (\ref{eq:classicfold}), respectively, we have  $\beta_i = h _ i \circ \varphi_i  \circ h_i^{-1}$ and $g_i = f_i \circ h_i^{-1} $. Thus $ g _ i $ is a $k_i$-fold such that $g _ i \circ \beta_i = g _ i.$ 
By Theorem \ref{teo: equivdiagrams1},   $(p_i)_s$ is equivalent to $(g_i)_s$. \end{proof}

        \quad

The normal form obtained in the previous result depends explicitly on the nontrivial eigenvalues of the reflections of $\Lambda_s$. However, replacing any $\alpha_i$ by a $j$-th iterate $\alpha_i^{j}$ of order $k_i$, for $j\in \{1,\ldots,k_i-1\}$, we obtain the same normal $g_i$. Indeed, this case yields another normal form $(\beta_1,\ldots, \beta_i^j, \ldots, \beta_s)$, where  
\[ \beta_i^j(x_1,\ldots,x_n)
=
\bigl(
x_1+\frac{a_{i1}(\lambda_i^{j}-1)}{\lambda_i-1}\,x_i,\;
\ldots,\;
\lambda_i^{j}x_i,\;
\ldots,\;
x_n+\frac{a_{in}(\lambda_i^{j}-1)}{\lambda_i-1}\,x_i \bigr). \]
Now, the normal form of their corresponding  divergent diagrams, say $(\tilde{g_i})_s$, comes from $\tilde{g_i} = f_i \circ \tilde{h}_i^{-1}$, where  
from a direct computation we find that $\tilde{h}_i = h_i$. Therefore, 
\[ \tilde{g_i} = f_i \circ \tilde{h}_i^{-1} =  f_i \circ h_i^{-1} = g_i.\]
Hence, we have just proved:

\begin{theorem} \label{thm: independence of generator}
The algebraic expression of the normal form~\eqref{eq: normalform_g} 
obtained in Proposition~\ref{prop:equivcomreflexoes} 
does not depend on the choice of generator of the cyclic group 
$\langle \alpha_i \rangle$, for any $i\in \{1,\ldots,s\}$.
\end{theorem}

We present in Theorem \ref{teo:class_Abelian} the classification of divergent diagrams of $\underline{k}$-folds associated with $s$-tuples of commuting reflections. This relies on Theorem~\ref{teo: equivdiagrams1} and on the following important property of simultaneous conjugacy: 

\begin{lemma} \label{prop:transversal} 
Transversality and being Abelian are invariants under equivalence.
\end{lemma}

\begin{theorem}\label{teo:class_Abelian}
	Let $\{\alpha_ 1, \ldots, \alpha_ s \}$ be a transversal set of linear reflections of orders $k_1, \ldots, k_s \geq 2$, respectively. Let $\lambda_i$ be the eigenvalue of $\alpha_i$, with $\lambda_i \neq 1$. If $\Lambda =  \langle \alpha_ 1, \ldots, \alpha_ s \rangle$ is an Abelian group, then the $s$-tuple $(\alpha_i)_s $ is equivalent to $ (\varphi_i)_s$, with $\varphi_i$ given in (\ref{eq:Abelian}), and a divergent diagram of $\underline{k}$-folds $(p_i)_s$ associated with $(\alpha_i)_s$ is equivalent to $(f_i)_s$, for $f_i$ given in (\ref{eq:classicfold}).
    \end{theorem}
\begin{proof} By Proposition \ref{prop:equivalenciadereflexões}, the $s$-tuple $(\alpha_i)_s$ is equivalent to $(\beta_i)_s$, with $\beta_i$ given in (\ref{eqpsi}). Since the reflections \(\beta_1,\ldots,\beta_s\) are diagonalizable and generate an Abelian group, they are simultaneously diagonalizable with respect to a basis \(\mathcal B\) of \(\mathbb C^n\). Using that \(\operatorname{codim}\operatorname{Fix}(\beta_i)=1\) and the transversality of \(\{\beta_1,\ldots,\beta_s\}\), 
it follows that, up to a possible reordering of ${\mathcal B}$,  \(\beta_i\) has the normal form \(\varphi_i\)  given in \eqref{eq:Abelian}. Hence \((\alpha_i)_s\) is equivalent to \((\varphi_i)_s\). The equivalence between \((p_i)_s\) and \((f_i)_s\) follows directly from Theorem~\ref{teo: equivdiagrams1}.
\end{proof}

\vspace{.3cm}

\subsection{Normal forms of divergent diagrams associated with pairs of reflections}\label{sections2nmaior}

This subsection is devoted to the case $s=2$. We provide normal forms for transversal pairs of linear reflections (Theorem \ref{prop: normal-form-pair-reflection}) and for the associated divergent diagrams (Theorem \ref{teoclassdobrassnm2}) on $(\mathbb{C}^n,0)$, for $n \geq 2$, extending to the complex setting the results established in Sections~6 and~7 of \cite{MMT} for the real case. Our approach relies on Theorem~\ref{teo: equivdiagrams1} and Theorem~\ref{thm: matrix}. \

Consider the transversal reflections $ \beta_{1_a} $ and $ \beta_{2_a} $ of Theorem~\ref{thm: matrix}, 
	\begin{equation}\label{eq:finalpsi1}
    \begin{array}{l}
			\beta_{1_a} (x) = (\lambda_1 x_1, x_2 + a_{12} x_1 , x_3 + a_{13} x_1 ,\ldots, x_n + a_{1n}x_1 ), \\
			\beta_{2_a} (x) = (x_1 + a_{21}x_2 , \lambda_2 x_2 , x_3 + a_{23} x_2, \ldots, x_n + a_{2n} x_2),
    \end{array}
	\end{equation} 
    where $x = (x_1, \ldots, x_n)$, $\lambda_i$ is a primitive $k_i$-th root of unity, $ a_{ij}\in \mathbb{C} $, for $ i=1,2 $, and $ 1 \leq j \leq n $, $ j \neq i $. We have that         
\begin{equation} \label{propaitrace}
{\rm tr}({\beta_{1_a} \circ \beta_{2_a}}) = a_{12} a_{21} + \lambda_1 + \lambda_2 + n-2.
\end{equation} 
We also have that $\beta_{1_a}$ and $\beta_{2_a}$ commute if only if $ a_{12}=a_{21}=0 $.\\

As a direct consequence from the above, when $n=2$ we have an equivalence relation on the space of para\-meters  $(a_{12}, a_{21}) \in \mathbb{C}^2 $ inherited from the equivalence of the corresponding reflections: $(a_{12}, a_{21})$ is equivalent to $(b_{12}, b_{21}) $ if and only if $ (\beta_{1_a}, \beta_{2_a}) $  and $ (\beta_{1_b}, \beta_{2_b}) $ are equivalent. In this case, $ \textrm{tr}(\beta_{1_a}\circ \beta_{2_a}) = a_ {12} a_ {21} + \lambda_1 + \lambda_2$, so by Theorem~\ref{thm: matrix}  we have the following representatives 
of $(a_{12}, a_{21}) $: 

\begin{equation*}
		\begin{cases}
			(0,0), \ \  \text{if the reflections commute}; \\
			(1,0), \ \   \text{if} \ a_{12}\neq 0 = a_{21};\\
			({\rm tr}(\beta_{1_a}\circ \beta_{2_a}) - \lambda_1 - \lambda_2 , 1), \ \ \text{if} \ a_{21}\neq 0.\\
		\end{cases}
	\end{equation*}

    \vspace{.3cm}

The next result is based on the stratification above and generalizes \cite[Proposition~7.1]{MMT}. In this way, the condition \(a_{12}a_{21}=4\), which was previously unexplained in the real setting, becomes clear, since it arises naturally from the identity \(a_{12}a_{21}= (\lambda_1-1)(\lambda_2-1)$, for $\lambda_1 = \lambda_2 = -1$.

 \vspace{.2cm}
 
\begin{theorem}\label{prop: normal-form-pair-reflection} For $ n \geq 2$, consider the pair $(\beta_{1_a}, \beta_{2_a})$ of transversal reflections on $(\mathbb{C}^n,0)$ of orders $k_1, k_2 \geq 2$ and eigenvalues $\lambda_1, \lambda_2 \neq 1$, respectively. If $\Lambda = \langle \beta_{1_a}, \beta_{2_a} \rangle$ is Abelian, then  $ (\beta _{1_a}, \beta_{2_a}) $ is equivalent to $ (\varphi_1, \varphi_{2}) $, where 
$$ \begin{array}{l}
\varphi_{1}(x) \ = (\lambda_1 x _ 1,x _ 2 ,\ldots , x _ n), \\ \varphi_{2}(x)  = (x_1, \lambda_2 x _ 2, x _ 3,\ldots , x _ n).
\end{array}$$ 
If $\Lambda$ is non-Abelian, then we have the following cases:
\begin{enumerate}
		\item[$(i)$] If $a_{12} a_{21} \neq (\lambda_1 - 1)(\lambda_2 - 1)$, then $(\beta_{1_a}, \beta_{2_a})$ is equivalent to $(\gamma_1, \gamma_2)$, where  \begin{equation}\label{eqsnm1}
        \begin{array}{l}
		\gamma_1(x) = (\lambda_1 x_1, x_1 + x_2, x_3 , \ldots , x_n), \\
        \gamma_2 (x)=(x_1, \lambda_2 x_2, x_3 , \ldots , x_n)
	\end{array}
    \end{equation} whenever $a_{21} = 0$, and 
\begin{equation}\label{eqsnm3}
\begin{array}{l}
\gamma_1(x) = \bigl(\lambda_1 x_1,  a_{12}a_{21}x_1 + x_2, x_3 , \ldots, x_n \bigr), \\ 
\gamma_2(x) =  (x_1 + x_2 , \lambda_2 x_2, x_ 3 , \ldots , x_n)
\end{array}
 \end{equation} whenever $a_{21}\neq 0$.

 \item[$(ii)$] If $n\geq 3$ and $a_{12} a_{21} = (\lambda_1 - 1)(\lambda_2 - 1)$, then $(\beta_{1_a}, \beta_{2_a})$ is equivalent to $(\gamma_1, \gamma_2)$, where
	\begin{equation}\label{eqsnm5}
    \begin{array}{ll}
		\gamma_1(x) = (\lambda_1 x_1, (\lambda_1 - 1)(\lambda_2-1)x_1 + x_2 , x_3 , \ldots, x_n ), \\ 
			\gamma_2(x) =  (x_1 + x_2 , \lambda_2 x_2, x_ 3 , \ldots , x_n)
             \end{array}
	\end{equation} whenever $(a_{23}, \ldots, a_{2n}) = \frac{a_{21}}{\lambda_1 - 1}(a_{13} , \ldots, a_{1n})$, and 
    \begin{equation} \label{eqsnm7}
    \begin{array}{l}
		\gamma_1(x) = (\lambda_1 x_1,(\lambda_1 - 1)(\lambda_2-1)x_1 + x_2 , x_3 , \ldots, x_n),\\ 
	   \gamma_2(x) =  (x_1 + x_2 , \lambda_2 x_2, x_2+ x_ 3 ,x_4, \ldots , x_n),
           \end{array}
	\end{equation} whenever $(a_{23}, \ldots, a_{2n}) \neq \frac{a_{21}}{\lambda_1 - 1}(a_{13} , \ldots, a_{1n})$.
    \end{enumerate}

\end{theorem}

\begin{proof}  The Abelian case is covered by Theorem~\ref{teo:class_Abelian}.
For the non-Abelian case, we have that $(\beta_{1_a}, \beta_{2_a})$ is equivalent to $(\beta_{1_b}, \beta_{2_b})$ if and only if there exists $C = (c_{jk})$ as in (\ref{Matrix-H}) and $\xi \neq 0$ such that $(b_{12}, b_{21})  = (\xi a_ {12}, \xi^{-1} a _ {21})$ and 	    
{\small \begin{equation*}
		b_{1j} = (\lambda_1 - 1) c_{j1}  + a_{12} c_{j2} + \sum_{k=3}^n a_{1k} c_{jk}, \quad b _ {2j} =\xi^{-1} \bigl((\lambda_2 - 1)  c_{j2} + a_{21} c_{j1}  + \sum _ {k=3}^n a_{2k} c_{jk}\bigr), \end{equation*}} \noindent with $3 \leq j \leq n $.    
        For the linear isomorphism induced by $\eta=(c_{jk})$, with $3 \leq j,k\leq n$, define the vector \[
		v_{\xi \eta} = ( \eta (a _ {13}, \ldots ,a _ {1n}), \xi^{-1} \eta(a _ {23}, \ldots ,a _ {2n})) \in  \mathbb{C}^{2n-4}
	\] and the auxiliary map $A_ {\xi \eta }  =  T_{v_{\xi \eta}} \circ L_{\xi}: \mathbb{C}^{2n-4}\rightarrow \mathbb{C}^{2n-4}$, where $T_{v_{\xi \eta}}$ is the translation by  $v_{\xi \eta}$ and \[
	L_{\xi}(z,w) = \bigl((\lambda_1 - 1) z + a _ {12}w, \xi^{-1} (a _ {21} z + (\lambda_2 - 1) w) \bigr),
	\] for all $(z,w) \in \mathbb{C}^{n-2} \times \mathbb{C}^{n-2}$. Notice that $L_{\xi}$ defines a linear isomorphism if and only if $a_{12} a_{21} \neq (\lambda_1-1)(\lambda_2-1)$. A direct computation shows that  $(\beta_{1_a} , \beta_{2_a})$ and $(\beta_{1_{b}}, \beta_{2_{b}})$ are equivalent if and only if
	\begin{equation}\label{eqrefc}
		(b _{12}, b _ {21}) = (\xi a_{12}, \xi^{-1}a_{21})\quad \text{and} \quad (b _ {13}, \ldots, b _ {1n} , b _ {23}, \ldots, b _ {2n}) \in {\rm Im} (A _ {\xi \eta}) ,
	\end{equation} for some $\xi \neq 0$.      
	 \begin{enumerate}
	 	\item[$ (i) $] If $a_ {12} a _ {21} \neq (\lambda_1-1)(\lambda_2-1)$, then ${\rm Im} (A_ {\xi \eta}) = \mathbb{C}^{2n-4}$ and we set $b_{jk}=0$ for all $j \in \{1,2\}$ and $k\in \{3,\ldots, n\}$.

If $a_{21} = 0$, then $a_{12}\neq 0$ (non-Abelian assumption) and setting $\xi = a_{12}^{-1}$ in (\ref{eqrefc}) we obtain $(b_{12}, b_{21}) = (1,0)$.  If $a_ {21} \neq 0$, we take $ \xi = a _ {21}$  so that
	$(b_{12}, b_{21})=(a_{12}a_{21},1)$. In this case, $\beta_{1_b} = \gamma_1$ and $\beta_{2_b} = \gamma_2$ as in (\ref{eqsnm1}) and (\ref{eqsnm3}), respectively.
      	
	\item[$ (ii) $] For $n \geq 3$, if $a _ {12} a _ {21} = (\lambda_1-1)(\lambda_2-1)$,  then  $v_ {\xi \eta}\in {\rm Im} (L_{\xi} )$ if and only if    \begin{equation}\label{eq:twotrianglesfinal}
		 (a _ {23} , \ldots , a _ {2n}) = \frac{a _ {21}}{\lambda_1 - 1} (a _ {13} ,  \ldots , a _ {1n}).
	\end{equation}
    In this case, $(0, \ldots, 0) \in {\rm Im}(A _ {\xi \eta}) = {\rm Im} (L_{\xi})$ and we can take $\xi = a _ {21}$ in (\ref{eqrefc})   so that $(b_{ 12 } , b _ {21}) = ((\lambda_1 -1)(\lambda_2 - 1) , 1 ),$  obtaining $\gamma_1$ and $\gamma_2$ as in (\ref{eqsnm5}).

    If $(a _ {23}, \ldots, a _ {2n}) \neq \frac{a_{21}}{\lambda_1 - 1} (a _ {13} , \ldots , a _ {1n} )$, then 
    $(0, 0, \ldots, 0,1, 0, \ldots, 0) \in {\rm Im} (A_{\xi \eta})$ for all $\eta$.  Choosing $\xi = a _ {21}$, $b_{23}=1$ and $b_{1k}=b_{2,k+1} = 0$ for all $k \in \{3,\ldots, n\}$, we obtain $\gamma_1$ and $\gamma_2$ as in (\ref{eqsnm7}).
	\end{enumerate} \end{proof}

\begin{remark} \label{Rem-Eing-Fix}
By  (\ref{propaitrace}), the equality $a_{12} a _ {21} = (\lambda_1 - 1)(\lambda_2-1)$ is equivalent to ${\rm tr}({\beta_{1_a} \circ \beta_{2_a}}) =  \lambda_1\lambda_2 + (n - 1)$. In this case, for $\beta_{1_a}$ and $\beta_{2_a}$ defined in \eqref{eq:finalpsi1}, the condition (\ref{eq:twotrianglesfinal}) can be replaced by $E_{\lambda_1}^{\beta_{1_a}} = E_{\lambda_2}^{\beta_{2_a}}$, where $E_{\lambda_i}^{\beta_{i_a}}$ denotes the eigenspace corres\-ponding to the eigenvalue $\lambda_i$ of $\beta_{i_a}$, $i=1,2$. Moreover, $a _ {21} = 0$ if and only if $E_{\lambda_2}^{\beta_{2_a}} \subseteq {\rm Fix}({\beta_{1_a}})$. These three conditions do not depend on the choice of class representative, as they are invariant by simultaneous conjugacy.  \end{remark}

We now use Theorem~\ref{prop: normal-form-pair-reflection}, rewritten in terms of Remark \ref{Rem-Eing-Fix}:

\begin{theorem}\label{teoclassdobrassnm2} For $n \geq 2,$ let $(p_1, p_2) : (\mathbb{C}^n,0) \rightarrow (\mathbb{C}^{2n}, 0) $ be a divergent diagram of $\underline{k}$-folds associated with the pair $(\alpha_1, \alpha_2)$ of transversal linear reflections on $(\mathbb{C}^n,0)$ of orders $k_1, k_2 \geq 2$ and eigenvalues $\lambda_1, \lambda_2 \neq 1$, respectively. Let  $\Lambda =  \langle \alpha_1 , \alpha_2 \rangle$.\begin{enumerate}
		\item[$(i)$] If $\Lambda$ is Abelian, then $(p_1, p_2)$ is equivalent to $(f_1, f_2)$, for $f_i$ as in (\ref{eq:classicfold}).  
               
		\item[$(ii)$] Suppose that $ \Lambda $ is non-Abelian and $ {\rm tr}({\alpha _ 1 \circ\alpha _ 2}) \neq \lambda_1\lambda_2 + (n - 1)$. If $ E_{\lambda_2}^{\alpha_2} \subseteq {\rm Fix}({\alpha _ 1})$,  then $(p_1, p_2)$ is equivalent to $ (g _1, g _2)$, where
		\begin{equation*}
		    \label{eq:g11} 
            \begin{array}{l}
            g _1 (x) =  \bigl(x_1^{k_1},   \frac{1}{1-\lambda_1}x_1 + x_2, x_3 , \ldots, x_ n\bigr), \vspace{0.1cm} \\
            g _2(x) = (x_1, x_2^{k_2}, x _ 3, \ldots , x_ n). 
		\end{array}
        \end{equation*}
		If $E_{\lambda_2}^{\alpha_{2}} \not\subseteq {\rm Fix}({\alpha_1})$,  then $(p_1, p_2)$ is equivalent to $ (g _1, g _2)$, where
		\begin{equation*} \begin{array}{l}
		    \label{eq:g13}
		g _1 (x) =  \bigl(x_1^{k_1},   \frac{{\rm tr}(\alpha_1 \circ \alpha_2) -  \lambda_1 - \lambda_2 + (2-n)}{1-\lambda_1}x_1 + x_2, x_3 , \ldots, x_ n\bigr), \vspace{0.2cm}\\ 
		     g _2(x) = \bigl(x_1 + \dfrac{1}{1-\lambda_2}x_2, x_2^{k_2}, x _ 3, \ldots , x_ n\bigr).
		\end{array}
		    		\end{equation*}
		\item[$ (iii) $] For $n\geq 3$, suppose that $ \Lambda $ is non-Abelian and $ {\rm tr}({\alpha_ 1 \circ \alpha_ 2}) =  \lambda_1\lambda_2 + (n - 1)$. If $E_{\lambda_1}^{\alpha_1} = E_{\lambda_2}^{\alpha_2}$, then $(p_1, p_2)$ is equivalent to $ (g_1, g_2)$, where \begin{equation*}
		    \label{eq:g15}
            \begin{array}{l}
            g _ 1 (x) = \bigl(x _ 1 ^{k_1}, (1-\lambda_2) x _ 1 + x _ 2, x _ 3 , \ldots, x _ n \bigr), \\ 
            g _ 2 (x) =  \bigl(x _ 1 + \frac{1}{1-\lambda_2} x _ 2 , x _ 2^{k_2}, x_ 3 , \ldots , x _ n \bigr).    
            \end{array}
 			\end{equation*} If $E_{\lambda_1}^{\alpha_1} \neq  E_{\lambda_2}^{\alpha_2}$, then $(p_1, p_2)$ is equivalent to $ (g_1, g_2)$, where
 		 \begin{equation*}
		    \label{eq:g17}
            \begin{array}{l}
            g _ 1 (x) = (x _ 1^{k_1},  (1-\lambda_2) x _ 1 + x_ 2, x _ 3 , \ldots, x _ n ), \vspace{0.1cm}\\ 
 		 	g _ 2 (x) =  \bigl(x _ 1 +\frac{1}{1-\lambda_2}  x _ 2 , x _ 2^{k_2},   \frac{1}{1-\lambda_2}x _ 2 + x_ 3 ,x _ 4, \ldots , x _ n\bigr).    
            \end{array}
 		 	 		 \end{equation*}
 	\end{enumerate} 
 \end{theorem}

\begin{proof} Again by Proposition \ref{prop:equivcomreflexoes}, a divergent diagram $(p_1, p_2) $ associated with $ (\alpha_1 , \alpha_2) $ is equiva\-lent to the divergent diagram $(g _1 , g_2) $ associated with $ (\beta_{1_a}, \beta_{2_a}) $, with
\[ \begin{array}{l}        
    g_1(x) = \bigl(x_1^{k_1}, \dfrac{a_ {12}}{1-\lambda_1} x_ 1 + x_2, \ldots , \dfrac{a_ {1n}}{1-\lambda_1} x_1 + x_ n \bigr), \vspace{0.2cm}\\ 
    g_2(x) = \bigl(x_1 + \dfrac{a_{21}}{1-\lambda_2}x_2, x_2^{k_2}, \dfrac{a_ {23}}{1-\lambda_2} x_2 + x_ 3, \ldots , \dfrac{a_ {2n}}{1-\lambda_2} x_2 + x_ n\bigr), \end{array} \]   
where $\lambda_i \neq 1$ is eigenvalue of $\alpha_i$ and $a_{ij} \in \mathbb{C}$ are given according to the reflections $\beta_{1_a}$ and $\beta_{2_a}$ defined in (\ref{eq:finalpsi1}). We recall that $(\beta_{1_a}, \beta_{2_a})$ are in the same equivalence class of $(\alpha_1 , \alpha_2)$. The result now follows by Theorem \ref{prop: normal-form-pair-reflection} and Remark \ref{Rem-Eing-Fix}.\\
\end{proof} 

\begin{example}
   Consider the divergent diagram of $(4,3)$-folds  $(p_1,p_2):(\mathbb{C}^3,0)\to (\mathbb{C}^6,0)$, 
   \[ \begin{array}{l}
   p_1(x_1, x_2, x_3) = \bigl( (x_1+x_3)^2 - x_ 3^4, x_3^4 + x_1 + x_3, 2 x_2 \bigr),  \\
   p_2(x_1, x_2, x_3) = (-x_2, x_3, x_1^3).
\end{array} \]
The fold $p_1$ is given by (\ref{eq: example}) in Example~\ref{Ex1-Fold-Ref}  and $p_2 = f_3 \circ h^{-1}$, where $f_3$ is the canonical fold (\ref{eq:classicfold}) for $i =3$, $k_3 = 3$ and $h(x_1,x_2,x_3)=(x_3, -x_1,x_2)$.  We then have $(p_1, p_2)$ associated with the non-commuting pair of reflections $(\alpha_1, \alpha_2)$, 
	$$\alpha_1(x_1,x_2,x_3)=(x_1+(1-i)x_3,\ x_2, \ ix_3),  \ \ \  \alpha_2(x_1,x_2,x_3)=(e^{\frac{2\pi i}{3}}x_1,x_2,x_3),$$ 
with $\lambda_1 = i$ and $\lambda_2 = e^{\frac{2\pi i}{3}}$.  
Their fixed-point subspaces are distinct coordinate planes, and so $(\alpha_1, \alpha_2)$ is transversal. We compute ${\rm tr}(\alpha_1 \circ \alpha_2)=e^{\frac{2\pi i}{3}}+i+1$, which is different from  $\lambda_1\lambda_2 + 2 = ie^{\frac{2\pi i}{3}} +2$. Moreover, $E_{\lambda_2}^{\alpha_2}$ is the $x_1$-axis. Hence, $(p_1, p_2)$ is in the class given in Theorem~\ref{teoclassdobrassnm2}($ii$), whose normal form is $(g_1, g_2)$, 
\[ \begin{array}{l}
g_1(x_1, x_2, x_3) = \bigl(x_1^4,\ \frac{1}{1-i}x_1+x_2,\ x_3 \bigr), \vspace{0.1cm}\\
g_2(x_1,x_2,x_3)=(x_1,\ x_2^3,\ x_3).
\end{array}\]

\end{example}

\end{document}